\RequirePackage{fix-cm}
\documentclass{svjour3}                     
\smartqed  
\usepackage{graphicx}
\usepackage{amsmath}

\usepackage{amsfonts}
\usepackage{amssymb}

\usepackage{booktabs} 
\usepackage{multirow}
\usepackage{extarrows}
\usepackage{matlab-prettifier}

\begin{document}

\title{Asymptotically Compatible Fractional Gr\"onwall Inequality and its Applications
}


\author{Daopeng Yin         \and
        Liquan Mei 
}


\institute{D. Yin \at
School of Mathematics and Statistics. Xi'an Jiaotong University. No.28, West Xianning Road, Xi'an, Shaanxi, 710049, P.R. China \\
           \and
           L. Mei \at
School of Mathematics and Statistics. Xi'an Jiaotong University. No.28, West Xianning Road, Xi'an, Shaanxi, 710049, P.R. China
\emph{Corresponding author}: \email{lqmei@mail.xjtu.edu.cn}
}

\date{Received: date / Accepted: date}

\maketitle

\begin{abstract}
In this work, we will give proper estimates for the discrete convolution complementary (DCC) kernels, which leads to the asymptotically compatible fractional Gr\"onwall inequality. The consequence can be applied in analysis of the stability and pointwise-in-time error of difference-type schemes on a non-uniform mesh. The pointwise error have explicit bound when a non-uniform time grid is given by specific scale function e.g. graded mesh, can be given directly.  Numerical experiments towards the conclusion of this work validate the error analysis.  
\end{abstract}
\section{Introduction}
In this paper, we provide asymptotically compatible fractional Gr\"owanll inequality as a useful tool for time-fractional evolution equations. The discrete complementary convolution (DCC) kernels were first introduced in the literature \cite{Liao2018,Liao2019a} to prove a class of fractional order inequalities. Due to the lack of proper estimation for DCC kernels, there are certain flaws in their asymptotic compatibility and optimal error estimation. Its asymptotic expression is an open problem up to now \cite{Yang2022} which will be rigorously proved in this work. Consequently, we provide pointwise-in-time error estimates of the non-uniform time mesh L1-scheme for solving time-fractional equations involving non-dissipation term under general regularity assumptions, based on the mentioned above results coupled with the discrete convolution summation (DCS) estimation. In such an analysis framework, this class of methods can be extended to other time difference-type schemes as well. Compared to the construction and analysis of barrier function, this route is considerably less expensive to analyze, yet our theoretical results concur with Kopteva's results \cite{Kopteva2020a}. Our main works are listed below
\begin{itemize}
\item The continuous counterpart of DCC will be given, and their coefficient depends of mesh scale and the ratio of adjacent grids was rigorously analyzed.
\item The asymptotically compatible fractional Gr\"onwall inequality will be proved which will be applied to the analysis of stability and error for time-fractional reaction-diffusion equations.
\item In providing grid functions, such as graded mesh, the explicit expression for pointwise-in-time error estimation can be strictly written based on the sharp estimate for DCS. 
\end{itemize}
As an extension of the parabolic problems, Caputo time-fractional subdiffusion equations have a wide range of applications. In this article, the following time-fractional order reaction-diffusion equation is taken into consideration. For \( \kappa >0 \), 
\begin{subequations}\label{eq:the main equation}
\begin{align}
\partial_{t}^{\alpha}u-\Delta u = & \kappa u + f(\mathbf{x},t), & \mathbf{x} \in \Omega \times (0,T],           \\
u(\mathbf{x},0)= & u_{0}(\mathbf{x}),                           & \mathbf{x} \in \Omega,                        \\
u(\mathbf{x},t)= & \phi(\mathbf{x},t),                                   & \mathbf{x} \in \partial \Omega \times [0,T],
\end{align}
\end{subequations}
where the  $\Omega \subset \mathbb{R}^{n}$, which is a $ n $-dimensional bounded Lipschitz domain
and the  time-fractional derivative utilizes the definition of Caputo-sense
\begin{equation}\label{eq:Caputo derivative}
\partial_t^\alpha u(\mathbf{x}, t) := \frac{1}{\Gamma(1-\alpha)} \int_0^t \frac{\partial u(\mathbf{x}, s)}{\partial s} \frac{1}{(t-s)^\alpha} d s, ~ 0<\alpha<1.
\end{equation}
When $ \alpha \to 1 $, the above-mentioned problem will be degenerated into a classical parabolic equation.
Typically denote   $ \omega _{\alpha}(s):= s^{\alpha-1}/\Gamma(\alpha)$ for $s\geq0$ which also be called Gelfand-Shilov function, and the above definition has an alternative convolution form
\begin{equation*}
\partial_t^\alpha u(\mathbf{x}, t)= \int_{0}^{t} \omega _{1-\alpha}(t-s)\partial_{s} u (\mathbf{x},s) \mathrm{d}s.
\end{equation*}
\par In time discretization, the $t_{k} $ for $ 1 \leq k \leq N$ represent the time discrete nodes on general non-uniform meshes and its $k$-th time step $\tau_{k} = t_{k} - t_{k-1}$.  The temporal semi-discrete scheme of continuous \eqref{eq:the main equation} is governed as
\begin{equation}\label{eq:time_discretization}
\mathrm{D}_{\tau} ^{\alpha} u^{n-\sigma}  = \Delta u^{n-\sigma} + \kappa u^{n-\sigma} +  f(\mathbf{x},t_{n-\sigma}).
\end{equation}
The discrete Caputo operator $ \mathrm{D}_{\tau} ^{\alpha} u^{n-\sigma} := \sum_{k=1}^n a_{n-k}^{(n)} \nabla_\tau u^k $, where the difference operator $\nabla_\tau u^{k}  := u^{k} - u^{k-1}$. Here the expression of discrete convolution (DC) kernels $a_{n-k}^{(n)}$ derived from specific discretization strategy e.g.
\begin{itemize}
\item the DC kernels via L1 scheme is
\[
a_{n-k}^{(n)}  := \frac{1}{\tau_k}\int_{t_{k-1}}^{t_k}  \omega_{1-\alpha}\left(t_n-s\right)\mathrm{d} s, ~ \text{ for }  1 \leq k \leq  n.
\]
\item the DC kernels via fast L1 scheme is 
\[
a_{n-k}^{(n)}  := 
\begin{cases}
\frac{1}{\tau_{n}}\int_{t_{n-1}}^{t_{n}}\omega_{1-\alpha}(t_{n}-s) \mathrm{d}s , &\text{ for } k=n, \\
\frac{1}{\tau_{k}}\int_{t_{k-1}}^{t_{k}} \sum\limits_{l=1}^{N_{q}} \mathrm{d}s, & \text{ for } 1\le k 
\le n-1.
\end{cases}
\]
\item the DC kernels via L21-$\sigma$ scheme is 
\[
a_{n-k}^{(n)}:=\mathtt{a}_{n-k}^{(n)} + r_{k-1}\mathtt{b}_{n-k+1}^{(n)} - \mathtt{b}_{n-k}^{(n)}, \text{ for } 1 \le k \le n,
\]
where 
\begin{align*}
\mathtt{a}_{n-k}^{(n)}  := & \frac{1}{\tau_{k}}\int_{t_{k-1}}^{\min\{t_{k}, t_{n-\sigma}\}} \omega_{1-\alpha}(t_{n-\sigma}-s)\mathrm{d}s, \text{ for } 1\le k \le n-1, \\
\mathtt{b}_{n-k}^{(n)}  := & \frac{2}{\tau_{k}(\tau_{k}+\tau_{k+1})}\int_{t_{k-1}}^{t_{k}} (s-t_{k-\frac{1}{2}})\omega_{1-\alpha}(t_{n-\sigma}-s)\mathrm{d}s, \text{ for } 1 \le k \le n-1,
\end{align*}
with setting $\mathtt{b}_{0}^{(n)} = \mathtt{b}_{n}^{(n)}\equiv 0$.
\end{itemize}
\par Our analysis focuses on the temporal discrete scheme, thus the spacial discretization can couple with finite element methods (FEMs) \cite{Ciarlet2002,Thomee2006}, finite difference methods (FDMs) \cite{Thomas1995}, spectral methods (SMs) \cite{Shen2011,Trefethen2000} and so on. Since without any restriction for time scales, the time adaptive strategy can used.
\par
There are some discussions on the regularity of the solutions to the subdiffusion equations in recent literature.
For the linear subdiffusion problem, Stynes et al \cite{Stynes2017} proved that while  $u_{0}\in D(\mathcal{L}^{5/2})$ and $\|f(\cdot, t)\|_{\mathcal{L}^{5 / 2}}+\| f_t(\cdot, t)\| _{\mathcal{L}^{1 / 2}}+t^\rho\| f_{t t}(\cdot, t)\| _{\mathcal{L}^{1 / 2}} \le C$, then
$\mid \partial_{t}^{m}u(\cdot,t)\mid \lesssim 1+t^{\alpha-m}, \text{ for } m = 0, 1, 2.$
Jin et al \cite{Jin2018}  analyzed that when
$u_{0}\in H_{0}^{1}(\Omega)\cap H^{2}(\Omega)$ and $\partial_{t}^{m}f\in (0,T;L^{2}(\Omega))$
then
$\|\partial_{t}^{m}\left(u(\cdot,t) - u_{0}\right) \| _{L^{2}(\Omega)} \lesssim t^{\alpha-m},m\geq 0.$
Al-Maskari et al
\cite{AlMaskari2019} suggested that the initial data
$u_0 \in \dot{H}^\nu(\Omega)$ can prove $\| \partial_t u(t)\| _{L^2(\Omega)} \lesssim t^{\nu \alpha / 2-1}$.
This weakly singular solution can also be a more generalized form, Zhang et al \cite{Zhang2022}
gave error analysis based on  a general assumption
$u(t)-u(0)=\sum_{k=1}^m \hat{u}_k t^{\beta_k}+\tilde{u}(t) t^{\beta_{m+1}}$, where $0<\beta_1<\cdots<\beta_m<\beta_{m+1}$ and $\tilde{u}(t) \in L^2([0, T]; X)$.
To this end, we give general assumptions on the regularity of the solution
\begin{equation*}\label{eq:regularity assumption}
\| \partial_t^m u \| \lesssim 1 + t^{\beta-m}, \text{ for } m =1, 2, \text{ and } \beta \in(0,1) \cup(1,2],
\end{equation*}
where $\beta$ is considered to be the regularity parameter in the general case \cite{Li2024}. When $\beta = \alpha$, it is the most common case.
\par 
The error order of a time discrete scheme is not constant, that is, the error order on various time levels has varying error orders, due to the non-local character of the Caputo derivative. Many academics have expressed concern about this aspect.
When $f(x,t,u)=0$, J. Gracia et al \cite{Gracia2017} analyzed them  numerical scheme have error result $\tau t_n^{\alpha-1}$.
Jin et al \cite{Jin2019} show that the time error of $L_{1}$ scheme should be
$\tau t_n^{-1}$ with $u_{0}\in L^{2}(\Omega)$.
These error estimates suggest that time factor  $t_{n}$ should be included in the error boundaries of time-fractional. Currently, there are two basic categories of analytical methods that may be applied to provide a result that is more cohesive. For one thing, N. Kopteva et al provided a general theoretical framework for pointwise-in-time error estimation of time-fractional diffusion problems using the discrete comparison principle which includes   $L1, L2, L21-\sigma$ schemes \cite{Kopteva2019,Kopteva2020,Kopteva2020a}  and also semi-linear problem \cite{Kopteva2020b}. For another,  Liao et al analyzed the error of the Caputo derivative by constructing DCC kernels and developing the fractional version of the discrete Gr\"onwall inequality \cite{Liao2018,Liao2019,Liao2019}, which is a popular tool for time evolution equations.
However, the pointwise-in-time error estimation under the general regularity of the solution using the L1  discrete scheme with non-uniform time mesh is still very sparse in related work under the frame of DCC.
\par
The numerical techniques may be stated as a DCS since the Fourier transform can be used to infer that the definition of the Caputo fractional order derivative originates from the continuous convolution form. Thus, the sharp estimation of DCS is a current topic of interest for related academics \cite{Stynes2023}.
There are certain works by C. Lubich \cite{Lubich1986,Lubich1988a,Lubich1988} that are linked to the investigation of this historical question. Numerous applications of this finding may be made to the research of numerical stability and convergence in time-fractional order problems.
\par The rest of the paper is organized as follows. In section \ref{sec:ACF_Gronwall}, we will prove the asymptotical estimate for DCC and then obtain the asymptotically compatible fractional Gr\"onwall inequality. In section \ref{sec:Error_analysis}, we present some applications, e.g. the numerical stability and pointwise-in-time error of L1 scheme of time-fractional parabolic equations. Section \ref{sec:numerical_verification} will show several numerical examples to confirm the theoretical results. There are two appendices at the end of this article including appendix \ref{sec:appendix_dcs} give a sharp estimation of DCS and rigorous proof and appendix \ref{sec:appendix_necessary_condition}  prove the necessary condition $ \frac{1}{2} \mathrm{D}_{\tau}^{\alpha}\|u^{n}\|^{2} \le (\mathrm{D}_{\tau}^{\alpha}u^{n}, u^{n}) $ with discrete Caputo operator.
\section{Asymptotically Compatible Fractional Gr\"onwall Inequality}\label{sec:ACF_Gronwall}
In the first subsection of this section, we will rigorously prove the asymptotic expression of DCC. The proof in uniform grids has been given by \cite[Theorem 3.3]{Jin2017}, but for general non-uniform grids, no one has yet completed. Our proof method is relatively concise and more general. Continuing with this result, we will provide the asymptotically compatible Gr\"onwall inequalities in the second subsection and their applications in stability and error analysis of semi-discrete schemes for time-fractional parabolic equations in the next section. 
\subsection{The asymptotic analysis of the DCC}
\par In this subsection, we provide an asymptotic estimate of DCC for general non-uniform mesh and its continuous counterpart will be used to prove the asymptotically compatible fractional Gr\"onwall inequality in the next subsection.
\par Firstly, we review the definition of DCC and express it in matrix form. The DCC defined by $(n-k)$-terms recurrence relations:
\begin{equation}\label{eq:recurrence relations of DCC}
p_{n-k}^{(n)}  := \frac{1}{a_{0}^{(n)}} 
\begin{cases}
1, &\text{ if } k = n, \\
\sum\limits_{j=k+1}^{n}p_{n-j}^{(n)}\left(a_{j-k-1}^{(j)} - a_{j-k}^{(j)}\right), & \text{ if } 1\le k <n.
\end{cases}
\end{equation}
The equivalent matrices form is
\begin{equation}
\mathbf{P}_{n}\mathbf{A}_{n}\mathbf{D}_{n} = \mathbf{I}_{n}, \label{eq:DCC_matrix_form}
\end{equation}
where
\begin{equation*}
\mathbf{P}_{n}  := 
\begin{bmatrix}
p_{0}^{(1)} &  0 & \ldots & 0     \\
p_{1}^{(2)} & p_{0}^{(2)} & \ldots & 0   \\
\vdots 	& \vdots & \ddots &  \vdots    \\
p_{n-1}^{(n)} & p_{n-2}^{(n)} & \ldots  & p_{0}^{(n)}
\end{bmatrix}, 
\mathbf{A}_{n}  := 
\begin{bmatrix}
a_{0}^{(1)} &  0 & \ldots  & 0 \\
a_{1}^{(2)} & a_{0}^{(2)} & \ldots & 0\\
\vdots 	& \vdots & \ddots & \vdots  \\
a_{n-1}^{(n)} & a_{n-2}^{(n)} & \ldots & a_{0}^{(n)}
\end{bmatrix} \nonumber ,
\end{equation*}
with the difference matrices
\begin{align*}
\mathbf{D}_{n} := 
\begin{bmatrix}
1 & & & \\
-1 & 1 & & \\
 & \ddots & \ddots   \\
& & -1 & 1
\end{bmatrix}
\text{ and identity matrix } \mathbf{I}_{n}.
\end{align*}
\par For all $1\le k \le n$, the relation $\sum_{j=k}^{n} p_{n-j}^{(n)} a_{j-k}^{(j)} = 1$ constantly holds is equivalent to 
\begin{align*}
\mathbf{P}_{n}\mathbf{A}_{n} = 
\begin{bmatrix}
1 & 0 & \ldots & 0 \\
1 & 1 & \ldots & 0 \\
\vdots & \vdots & \ddots & \vdots \\ 
1 & 1 & \ldots & 1  
\end{bmatrix} = \mathbf{D}_{n}^{-1},
\end{align*}
which can obtain from \eqref{eq:DCC_matrix_form} directly.
\par Before starting the discussion, it is necessary to declare some notations in this article, the adjoint time step ratio $\rho=\max_{1\le n \le N-1}\{\frac{\tau_{n+1}}{\tau_{n}}\},$ and maximum time step scale $ \tau=\max_{1\le n \le N}\{ \tau_{n}\}$.  The asymptotic expression of DCC is 
\begin{equation}\label{key}
\tilde{p}_{n-k}^{(n)} := \int_{t_{k-1}}^{t_{k}} \omega_{\alpha}( t_{n} - s )\mathrm{d}s. \nonumber
\end{equation}
\begin{lemma}\label{lemma:quotient of similar coefficient}
For all $1 \le k+1 \le j \le n $, 
\begin{equation}
\max_{k+1\le j \le n}\left\{ \frac{\tilde{p}_{n-j}^{(n)}\left(a_{j-k-1}^{(j)} - a_{j-k}^{(j)}\right)}{-\int_{t_{j-1}}^{t_{j}}\omega_{\alpha}(t_{n} - t)\int_{t_{k-1}}^{t_{k}}\omega^{\prime}_{1-\alpha}(t-s)\mathrm{d}s\mathrm{d}t} \right\} := C_{\rho, \tau} \le \frac{\rho+1}{2} + \mathcal{O}(\tau). \nonumber
\end{equation}
\end{lemma}
\begin{proof}
For convenience, we denote 
\begin{align*}
q_{n-k,j}^{(n)} := & \frac{\tilde{p}_{n-j}^{(n)}\left(a_{j-k-1}^{(j)} - a_{j-k}^{(j)}\right)}{-\int_{t_{j-1}}^{t_{j}}\omega_{\alpha}(t_{n} - t)\int_{t_{k-1}}^{t_{k}}\omega^{\prime}_{1-\alpha}(t-s)\mathrm{d}s\mathrm{d}t},
\end{align*}
which illustrate the similarity between $\tilde{p}_{n-j}^{(n)}(a_{j-k-1}^{(j)} - a_{j-k}^{(j)})$ and the denominator part.
Equivalently,  
\begin{align*}
q_{n-k,j}^{(n)} = \frac{\frac{1}{\tau_j}\tilde{p}_{n-j}^{(n)} \cdot \frac{1}{\tau_k}\left(a_{j-k-1}^{(j)} - a_{j-k}^{(j)}\right)}{-\frac{1}{\tau_{j}}\int_{t_{j-1}}^{t_{j}}\omega_{\alpha}(t_{n} - t)\cdot\frac{1}{\tau_{k}}\int_{t_{k-1}}^{t_{k}}\omega^{\prime}_{1-\alpha}(t-s)\mathrm{d}s\mathrm{d}t}.
\end{align*}
Notice that 
\begin{align*}
\frac{\tilde{p}_{n-j}^{(n)}}{\tau_{j}} = \int_{0}^{1}\omega_{\alpha}(t_{n}-t_{j}+\theta\tau_{j})\mathrm{d}\theta = \omega_{\alpha}(t_{n}-t_{j}) + \mathcal{O}(\tau_{j}),
\end{align*}
\begin{align*}
\frac{a_{j-k-1}^{(j)} - a_{j-k}^{(j)}}{\tau_{k}} = & \frac{1}{\tau_{k}}\int_{0}^{1} \sum_{m=1}^{\infty} \frac{\omega^{(m)}_{1-\alpha} (t_{j}-t_{k+1})}{m!}[(\theta\tau_{k+1})^{m} - (\tau_{k+1}+\theta\tau_{k})^{m}]\mathrm{d}\theta \\
= & -\frac{1}{2}\omega^{\prime}_{1-\alpha}(t_{j}-t_{k+1})(r_{k+1}+1) + \mathcal{O}(\tau_{k}),
\end{align*}
and 
\begin{align*}
& - \frac{1}{\tau_{j}}\int_{t_{j-1}}^{t_{j}}\omega_{\alpha}(t_{n}-t)\cdot \frac{1}{\tau_{k}}\int_{t_{k-1}}^{t_{k}}\omega^{\prime}_{1-\alpha}(t-s)\mathrm{d}s \mathrm{d}t \\
= & - \int_{0}^{1}\omega_{\alpha}(t_{n}-t_{j} + (1-\varphi)\tau_{j}) \int_{0}^{1}\omega^{\prime}_{1-\alpha}(t_{j} - (1-\varphi)\tau_{j}-t_{k+1}+\tau_{k+1}+\theta\tau_{k})\mathrm{d}\theta \mathrm{d}\varphi .
\end{align*}
Then
\begin{align*}
q_{n-k,j}^{(n)} = & \frac{\left(\omega_{\alpha}(t_{n}-t_{j})+\mathcal{O}(\tau_{j})\right)\cdot\left(-\frac{1}{2}\omega_{1-\alpha}^{\prime}(t_{j}-t_{k+1})(r_{k+1}+1)+\mathcal{O}(\tau_{k})\right)}{-\left(\omega_{\alpha}(t_{n}-t_{j})+\mathcal{O}(\tau_{j})\right)\cdot\left(\omega_{1-\alpha}^{\prime}(t_{j}-t_{k})+\mathcal{O}(\tau_{j})+\mathcal{O}(\tau_{k+1})+\mathcal{O}(\tau_{k})\right)} \\
= & 
\frac{r_{k+1}+1}{2} + \mathcal{O}(\tau_{j})+\mathcal{O}(\tau_{k+1})+\mathcal{O}(\tau_{k}).
\end{align*}
Therefore, $q_{n-k,j}^{(n)} =  (r_{k+1}+1)/2 + \mathcal{O}(\tau_{j}) + \mathcal{O}(\tau_{k+1}) + \mathcal{O}(\tau_{k})$ and it follows that  $q_{n-k,j} ^{(n)} \le (\rho+1)/2 + \mathcal{O}(\tau)$. 
\end{proof}
In practice, we need to bound the maximum mesh scales and maximum ratio of adjacent meshes as a prior assumption. 
This hypothesis is reasonable and intuitive, for which the sufficiently uniform and sufficiently fine mesh implies the discrete problem becomes closer to the continuous one, and at the same time, the DCC will close to its continuous counterpart.
\begin{lemma}[The asymptotic analysis of DCC]
For all $1\le k \le n$, the asymptotic expression $\tilde{p}_{n-k}^{(n)} = \int_{t_{k-1}}^{t_{k}} \omega_{\alpha}( t_{n} - s )\mathrm{d}s$ of DCC yields that 
\begin{equation}
\label{eq:the estimates of DCC kernels}
p_{n-k}^{(n)} \le \tilde{p}_{n-k}^{(n)},
\end{equation}
\end{lemma}
\begin{proof}
From the construction of DCC kernels which want to illustrate the discontinuous semi-group property and the semi-group property on each right part of the following relations, we know that
\begin{subequations}
\begin{align}
\sum_{j=k}^{n} p_{n-j}^{(n)} a_{j-k}^{(j)} = & 1= \sum_{j=k}^{n} \int_{t_{j-1}}^{t_{j}} \omega_{\alpha}(t_{n}-s) \omega_{1-\alpha}(s-t_{k-1}) \mathrm{d}s, \label{eq:key_equation_A}\\
\sum_{j=k+1}^{n} p_{n-j}^{(n)} a_{j-k-1}^{(j)} = & 1= \sum_{j=k+1}^{n} \int_{t_{j-1}}^{t_{j}} \omega_{\alpha}(t_{n}-s) \omega_{1-\alpha}(s-t_{k}) \mathrm{d}s. \label{eq:key_equation_B}
\end{align}
\end{subequations}
Subtract equation \eqref{eq:key_equation_A} over \eqref{eq:key_equation_B}, then for all $k<n$ has 
\begin{subequations}
\begin{align}
0 = & p_{n-k}^{(n)}a_{0}^{(k)} - \sum_{j=k+1}^{n}p_{n-j}^{(n)}\left(a_{j-k-1}^{(j)} - a_{j-k}^{(j)}\right) \label{eq:DCC_equation} \\
= & \int_{t_{k-1}}^{t_{k}} \omega_{\alpha}(t_{n} - t)\omega_{1-\alpha}(t-t_{k-1}) \mathrm{d}t \nonumber \\ 
& - \sum_{j=k+1}^{n}\int_{t_{j-1}}^{t_{j}}\omega_{\alpha}(t_{n} - t)\left[\omega_{1-\alpha}(t-t_{k}) - \omega_{1-\alpha}(t-t_{k-1})\right]\mathrm{d}t \label{eq:continuous_DCC_equation}	\\
= & \left(q_{n-k}^{(n)}\tilde{p}_{n-k}^{(n)}\right)a_{0}^{(k)} - \sum_{j=k+1}^{n} \left(q_{n-j}^{(n)}\tilde{p}_{n-j}^{(n)}\right)\left(a_{j-k-1}^{(j)} - a_{j-k}^{(j)}\right), \label{eq:ratio_equation}
\end{align}
\end{subequations}
where
\begin{align*}
q_{n-k}^{(n)} := 
\begin{cases}
1/\tilde{p}_{0}^{(n)}a_{0}^{(n)}, &\text{ if } k = n, \\
\sum\limits_{j=k+1}^{n} \dfrac{\left(q_{n-j}^{(n)}\tilde{p}_{n-j}^{(n)}\right)\left(a_{j-k-1}^{(j)} - a_{j-k}^{(j)}\right)}{\tilde{p}_{n-k}^{(n)}a_{0}^{(k)}} , &\text{ if } k < n.
\end{cases}
\end{align*}
Each of the above three relations has its individual effect:
\begin{itemize}
\item
\eqref{eq:DCC_equation} comes from the process of constructing DCC and actually has similar semi-group properties in discrete cases
\item
\eqref{eq:continuous_DCC_equation} is the significant process of finding the continuous corresponding one of DCC.  The analysis here requires a bit of intuition to observe the properties of continuous semi-groups and discrete analogs, which have strong similarities in their constituent forms In the subsequent proof, what we need to do is to quantify the similarity.
\item
\eqref{eq:ratio_equation} can exist independently of the previous two relations, and its purpose to characterize the ratios between $p _{n-k}^{(n)}$ and $\tilde{p}_{n-k}^{(n)}$ via unknown parameters $q _{n-k}^{(n)}$ for all $1\le k \le n$. From the matrix form \eqref{eq:DCC_matrix_form}, it can be seen that because the monotonicity of $a _{n-k}^{(n)}$ indicates that the $\mathbf{A}_{n}\mathbf{D}_{n}$ matrix is an M-matrix, then it can be determined that $\mathbf{P}_{n}$ is a positive matrix. Obviously, since $\tilde{p}_{n-k}^{(n)}$ is positive through its specific expressions, and then $q _{n-k}^{(n)}$ is also a positive number. The definition of this ratio may be similar to the definition of DCC, but there are differences in specific expressions. The process is to define the value of $k$ through the term $k$ to $n$. By repeating this, all the values of $q _{n-k}^{(n)}$ can be defined.
\end{itemize}
\par Then the matrix form of equations \eqref{eq:DCC_equation} and \eqref{eq:continuous_DCC_equation} is 
\begin{equation*}
\mathbf{I}_{n} = \mathbf{P}_{n}\mathbf{A}_{n}\mathbf{D}_{n} = \left(\mathbf{Q}_{n} \circ \widetilde{\mathbf{P}}_{n} \right) \mathbf{A}_{n}\mathbf{D}_{n},
\end{equation*}
where $\mathbf{Q}_{n} \circ \widetilde{\mathbf{P}}_{n}$ is the Hadamard product of matrices and
\begin{equation}
\mathbf{Q}_{n} := 
\begin{bmatrix}
q_{0}^{(1)} & 0 & \ldots & 0 \\
q_{1}^{(2)} & q_{0}^{(2)} & \ldots & 0 \\
\vdots 	& \vdots & \ddots & \vdots  \\
q_{n-1}^{(n)} & q_{n-2}^{(n)} & \ldots & q_{0}^{(n)}
\end{bmatrix} \nonumber .
\end{equation}
 Obviously, the inverse of matrices $\mathbf{A}_{n}\mathbf{D}_{n}$ exists and $\left(\mathbf{A}_{n}\mathbf{D}_{n}\right)^{-1} = \mathbf{P}_{n}$, thus $0 < \mathbf{P}_{n} = \mathbf{Q}_{n} \circ \widetilde{\mathbf{P}}_{n}$, where the M-matrices property of $\mathbf{A}_{n}\mathbf{D}_{n}$ guarantees that  $\mathbf{P}_{n}$ is a positive  matrix.
The rest of this proof will discuss the uniform upper bound of $q_{n-k}^{(n)}$ from its definition,
\begin{align*}
q_{n-k}^{(n)} = &  \sum_{j=k+1}^{n} \frac{ \left(q_{n-j}^{(n)}\tilde{p}_{n-j}^{(n)}\right)\left(a_{j-k-1}^{(j)} - a_{j-k}^{(j)}\right)}{\tilde{p}_{n-k}^{(n)}a_{0}^{(k)}} \\
\le &  \sum_{j=k+1}^{n} \frac{ \left(q_{n-j}^{(n)}\tilde{p}_{n-j}^{(n)}\right)\left(a_{j-k-1}^{(j)} - a_{j-k}^{(j)}\right)}{\int_{t_{k-1}}^{t_{k}} \omega_{\alpha}(t_{n} - t)\omega_{1-\alpha}(t-t_{k-1}) \mathrm{d}t} \\
\le & C_{r,\tau}\sum_{j=k+1}^{n} q_{n-j}^{(n)} \frac{\int_{t_{j-1}}^{t_{j}}\omega_{\alpha}(t_{n} - t)\left[\omega_{1-\alpha}(t-t_{k}) - \omega_{1-\alpha}(t-t_{k-1})\right]\mathrm{d}t}{\int_{t_{k-1}}^{t_{k}} \omega_{\alpha}(t_{n} - t)\omega_{1-\alpha}(t-t_{k-1}) \mathrm{d}t},
\end{align*}
where   the first inequality from  the Chebyshev sorting inequality  
\begin{align*}
& \int_{t_{k-1}}^{t_{k}} \omega_{\alpha}(t_{n} - t)\omega_{1-\alpha}(t-t_{k-1}) \mathrm{d}t \\
\le & \int_{t_{k-1}}^{t_{k}} \omega_{\alpha}(t_{n} - t) \mathrm{d}t \cdot \frac{1}{\tau_{k}}\int_{t_{k-1}}^{t_{k}} \omega_{1-\alpha}(t-t_{k-1}) \mathrm{d}t = \tilde{p}_{n-k}^{(n)}a_{0}^{(k)}. 
\end{align*}
and the second one from the maximum of non-negative numbers
\begin{align*}
C_{r,\tau} := & \max_{k+1\le j \le n}\left\{ \frac{\tilde{p}_{n-j}^{(n)}\left(a_{j-k-1}^{(j)} - a_{j-k}^{(j)}\right)}{-\int_{t_{j-1}}^{t_{j}}\omega_{\alpha}(t_{n} - t)\int_{t_{k-1}}^{t_{k}}\omega^{\prime}_{1-\alpha}(t-s)\mathrm{d}s\mathrm{d}t} \right\}.
\end{align*}
\eqref{eq:continuous_DCC_equation} implies 
\begin{equation*}
 \sum_{j=k+1}^{n} \frac{\int_{t_{j-1}}^{t_{j}}\omega_{\alpha}(t_{n} - t)\left[\omega_{1-\alpha}(t-t_{k}) - \omega_{1-\alpha}(t-t_{k-1})\right]\mathrm{d}t}{\int_{t_{k-1}}^{t_{k}} \omega_{\alpha}(t_{n} - t)\omega_{1-\alpha}(t-t_{k-1}) \mathrm{d}t} \equiv 1,  
\end{equation*}
and then the discrete Gr\"onwall inequality \cite{Clark1987} with its $a_n=\varepsilon\to 0$ with loss of generality,  shows that 
\begin{align*}
q_{n-k}^{(n)} \le \varepsilon (1+C_{r,\tau})^{n-k} \le 1,
\end{align*}
where the estimation of $C_{r,\tau}$ has been analyzed in lemma \ref{lemma:quotient of similar coefficient}.
\end{proof}
\begin{remark}\label{rem:}
When  $r_{n+1}\equiv 1$ and $\tau\to 0$ for all $1\le n \le N-1$, then $q_{n-j,k}^{(n)}\to 1$ for all $1\le k+1\le j \le n \le N$. Meanwhile, 
\begin{align*}
& q_{n-k,k}^{(n)}\tilde{p}_{n-k}^{(n)}a_{0}^{(k)} - \sum_{j=k+1}^{n} q_{n-j,k}^{(n)}\tilde{p}_{n-j}^{(n)}\left(a_{j-k-1}^{(j)} - a_{j-k}^{(j)}\right) \\
\xrightarrow{~\tau\to 0 ~} ~ & \tilde{p}_{n-k}^{(n)}a_{0}^{(k)} - \sum_{j=k+1}^{n} \tilde{p}_{n-j}^{(n)}\left(a_{j-k-1}^{(j)} - a_{j-k}^{(j)}\right),
\end{align*}
namely $p_{n-k}^{(n)} \to \tilde{p}_{n-k}^{(n)}$ from entries-wise congruence of matrix.  From the point view, when 
 $q_{n-j,k}^{(n)}\to 1$,  the $q _{n-k}^{(n)}$ bounded by the convex combination of term from $k+1$ to $n$, 
and then the above inequality can be optimized.
\end{remark}
\begin{figure}[h]
\centering
\includegraphics[width=2.in]{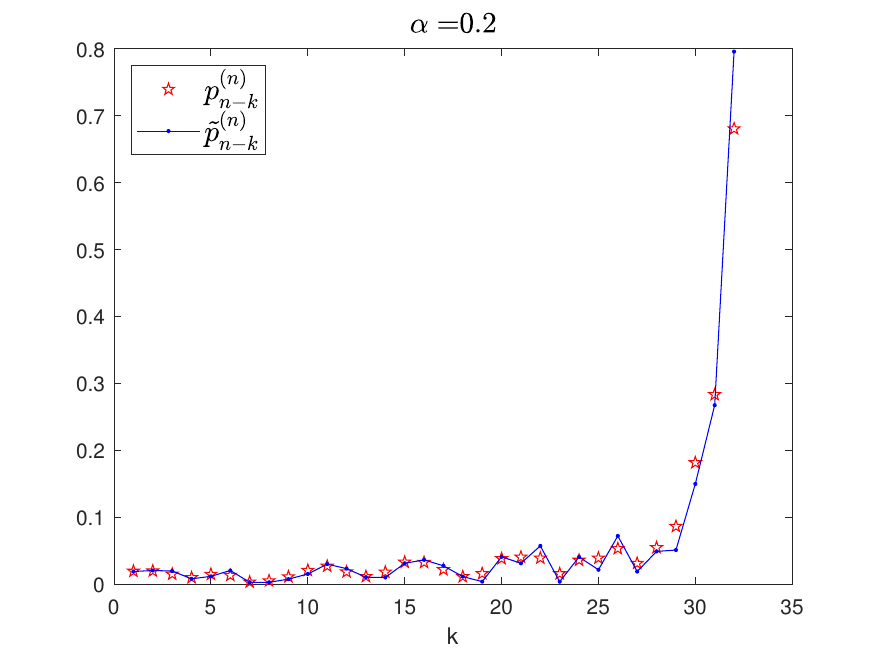}
\includegraphics[width=2.in]{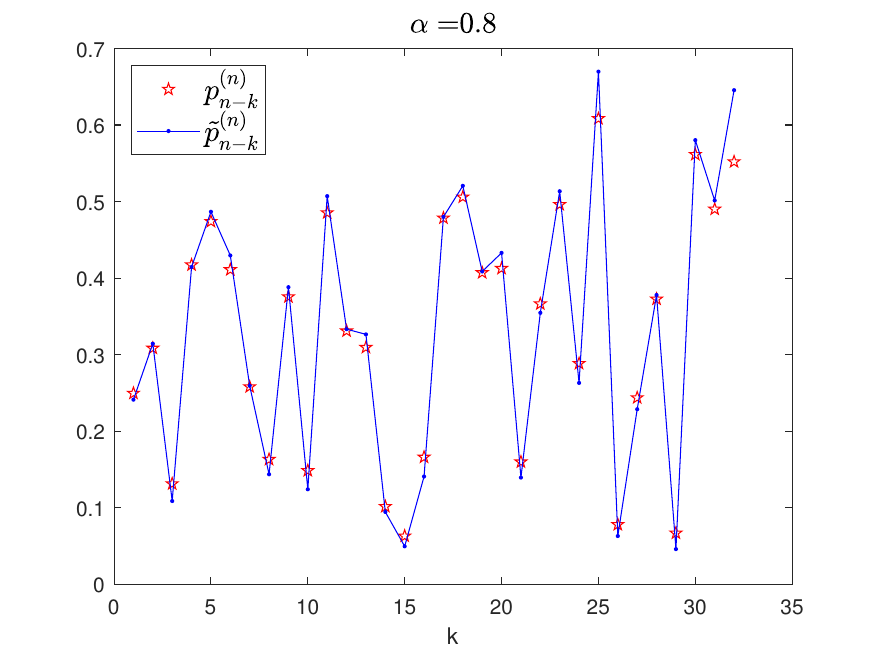}
\caption{The diagram of the $p_{n-k}^{(n)}$ and $\tilde{p}_{n-k}^{(n)}$ with fixed $n=32$}
\label{fig:DCC}
\end{figure}
\begin{remark}\label{rem:}
\par Unlike usual habits, the construction of DCC defines $k$-th term  from the larger part of indicator $n$ to $k+1$. The construction process of DCC tells us that the right-side part of each sub-figure in the figure \ref{fig:DCC} is the initial defined part, and the approximation of asymptotic estimation is not very significant in this part. The  left-side parts, since the more terms are associated, require asymptotic estimation the most.
\par In fact, the mesh ratio $r$ illustrates the degree of mesh smoothness. If time scales $\tau_k$ satisfies some proper smooth function with respect to \( k \), e.g. $\tau_k=0.4\sin(\frac{3k\pi}{n})+0.41$, the graph of $p _{n-k}^{(n)}$ was basically identical with the $\tilde{p}_{n-k}^{(n)} $ in the left sub-figure of figure \ref{fig:quotient_DCC}.  
\end{remark}
\begin{figure}[h]
\centering
\includegraphics[width=2.in]{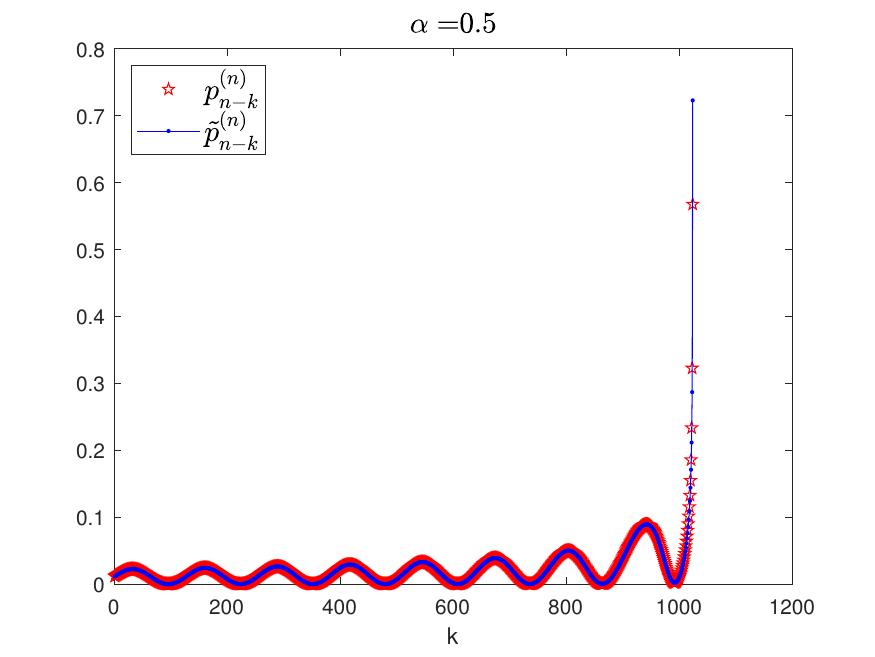}
\includegraphics[width=2.in]{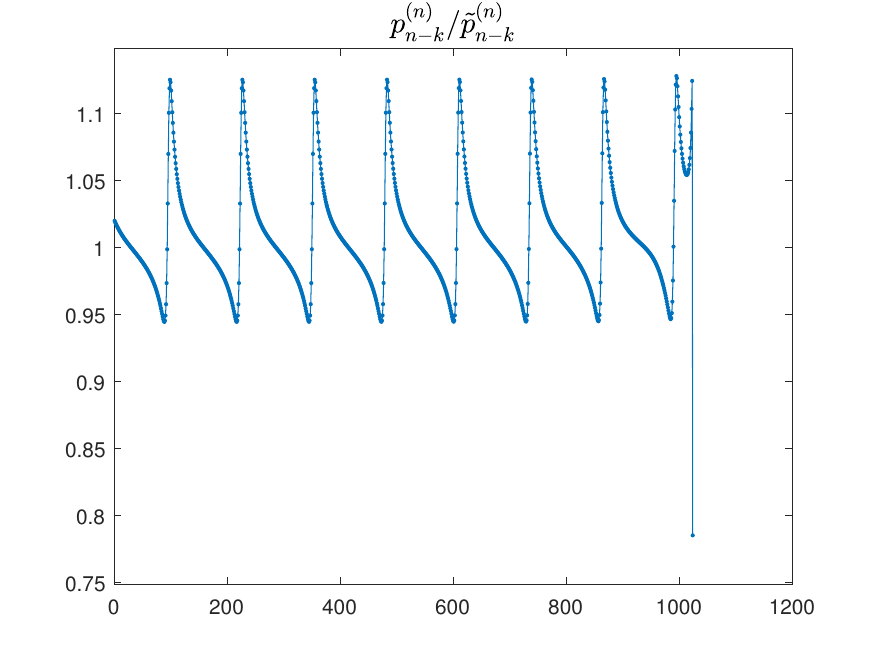}
\caption{The graphs of $p_{n-k}^{(n)}$ and $\tilde{p}_{n-k}^{(n)}$ with fixed $n=1024$.}
\label{fig:quotient_DCC}
\end{figure}
\subsection{The Asymptotically compatible fractional Gr\"onwall inequality}
\begin{theorem}
\label{thm:fractional_Gronwall_inequality_theorem}
For arbitrary positive constant $\kappa$, the non-negative sequence $\{V_{k}\}_{k=0}^{N}$ and $\{F_{k}\}_{k=0}^{N}$, 
\begin{equation}
 \mathrm{D}_{\tau}^{\alpha}V_{n} \le  \kappa V_{n} + F_{n}, \nonumber
\end{equation}
then 
\begin{equation}
V_{n} \le E_{\alpha}(\kappa t_{n}^{\alpha})\left(V_{0}+\max\limits_{1\le n \le N}\left\{\sum\limits_{j=1}^n \int_{t_{j-1}}^{t_{j}}\omega_{\alpha}(t_{n}-s)\mathrm{d}s \cdot F_{j}\right\}\right). \nonumber
\end{equation}
\end{theorem}
\begin{proof}
Adjust the index $n$ as $j$ of original inequality and take $\Sigma_{j=1}^n p_{n-j}^{(n)}$ on both sides,
\begin{equation*}
\sum_{j=1}^n p_{n-j}^{(n)} \sum_{k=1}^j a_{j-k}^{(j)} \nabla_\tau V_{k} \leq \kappa\sum_{j=1}^{n} p_{n-j}^{(n)}V_{j} +
\sum_{j=1}^n p_{n-j}^{(n)}F_{j},
\end{equation*}
then 
\begin{align*}
V_{n} \le & V_{0} +  \kappa \sum_{j=1}^n p_{n-j}^{(j)}V_{j} +  \sum_{j=1}^n p_{n-j}^{(n)}F_{j} \le V_{0} + \kappa \sum_{j=1}^n \tilde{p}_{n-j}^{(j)}V_{j} + \sum_{j=1}^n \tilde{p}_{n-j}^{(n)}F_{j}.
\end{align*} 
The classical Gr\"onwall inequality yields that 
\begin{equation}
V_{n} \le E_{\alpha}(\kappa t_{n-1}^{\alpha})\left(V_{0}+\max\limits_{1\le \nu \le n}\left\{\sum\limits_{j=1}^\nu \int_{t_{j-1}}^{t_{j}}\omega_{\alpha}(t_{n}-s)\mathrm{d}s \cdot F_{j}\right\}\right), \nonumber
\end{equation}
This completes the proof.
\end{proof}
\begin{corollary}
Comparing with classical Gr\"onwall inequality that
\begin{align*}
\int_{t_{j-1}}^{t_{j}}\omega_{\alpha}(t_{n}-s)\mathrm{d}s \xrightarrow{\alpha \to 1 } \tau_{j},
\end{align*}
the theorem \ref{thm:fractional_Gronwall_inequality_theorem} gives an asymptotically compatible result that
\begin{align*}
& 
\begin{cases}
\text{ If } \mathrm{D}_{\tau}^{\alpha}V_{n} \le  \kappa V_{n}+F_{n}, \text{ then } \\
V_{n} \le E_{\alpha}(\kappa t_{n-1}^{\alpha})\left(V_{0} + \max\limits_{1\le \nu \le n}\left\{\sum\limits_{j=1}^\nu \int_{t_{j-1}}^{t_{j}}\omega_{\alpha}(t_{n}-s)\mathrm{d}s \cdot F_{j}\right\}\right),
\end{cases} \\
\xrightarrow{\alpha \to 1 } & 
\begin{cases}
\text{ If } \delta_{\tau} V_{n} \le \kappa V_{n} + F_{n}, \text{ then } \\
V_{n} \le \exp(\kappa t_{n-1})\left(V_{0}+\sum\limits_{j=1}^n \tau_{j}F_{j}\right).
\end{cases}
\end{align*}
\end{corollary}
\section{Error analysis}\label{sec:Error_analysis}
In this section, we will provide applications of Gr\"onwall inequalities in numerical analysis.  Since FEMs, Galerkin SMs, and energy estimation for continuous problems, they are all analyzed under the $ L^2 $-norm. Therefore, in this section, we mainly consider estimation and analysis under the $ L^2 $-norm or discrete $ L^{2} $-norm.
\begin{lemma}[$L^2$-stability]\label{thm:L2_stability_2}
The temporal semi-discretization \eqref{eq:time_discretization} have 
\begin{equation}\label{eq:L^2-stability 2}
\left\| u^{n} \right\| \le E_{\alpha}(\kappa Ct_{n-1}^{\alpha})\left(\left\| u^{0} \right\| + C \max\limits_{1\le \nu \le n}\left\{\sum\limits_{j=1}^\nu \int_{t_{j-1}}^{t_{j}}\omega_{\alpha}(t_{n}-s)\mathrm{d}s \cdot \| f(\mathbf{x},t_{j}) \| \right\}\right). \nonumber
\end{equation}
\end{lemma}
\begin{proof}
Directly use theorem \ref{thm:fractional_Gronwall_inequality_theorem} and the monotonicity of discrete convolution kernel, this proof can be done. 
\end{proof}
Denote \( u_h(t_n)  \) the spacial semi-discretization solution at time \( t_n \) for all  $n \in \{ 1,2,\ldots,N\}$, and the truncation error derived from temporal discrete \( 	\left| \mathrm{D}_\tau^\alpha u_h^n - \partial_t^\alpha u_h(t_n) \right| = | \sum_{k=1}^{n} T_{n-k}^{(n)} | =: |  \mathtt{T}_{\tau}^{n} | \) with
\begin{equation}
T_{n-k}^{(n)} := \int_{t_k-1}^{t_{k}}\omega_{1-\alpha}\left(t_n-s\right)\left[\frac{\nabla_{\tau}u_{m}^{k}}{\tau_{k}}-\frac{\partial u}{\partial s}\left(x_m, s\right)\right] d s.  \nonumber
\end{equation}
\begin{lemma}[{Lemma 3.1} \cite{Liao2019}]
\label{lem:truncated_error}
Under the assumption $	\| \partial_t^m u \| \lesssim 1 + t^{\beta-m}, \text{ for } m =1, 2, \text{ and } \beta \in(0,1) \cup(1,2]$,	the truncated error on  general non-uniform meshes gives
\begin{equation}\label{eq:truncated error on non-uniform meshes}
|\mathtt{T}_{\tau}^{n}| \leq	a_0^{(n)} G^n + \sum_{k=1}^{n-1}\left(a_{n-k-1}^{(n)}-a_{n-k}^{(n)}\right) G^k, \quad n \geq 1.
\end{equation}
where $G^k :=2 \int_{t_{k-1}}^{t_k}\left(t-t_{k-1}\right)\left|u^{\prime \prime}(t)\right| \mathrm{d} t$. Since the positivity of $p_{n-j}^{(n)}$,  then 
\begin{align*}
\sum_{j=1}^n p_{n-j}^{(n)}\left|\mathtt{T}_{\tau}^j\right| \leq 2 \sum_{k=1}^n p_{n-k}^{(n)} a_0^{(k)} \int_{t_{k-1}}^{t_k}\left(t-t_{k-1}\right)\left|u_{t t}\right| \mathrm{d} t,
\end{align*}
which is asymptotically compatible with the global truncation error when $\alpha \to 1^{-}$.
\end{lemma}
\begin{remark}
If we want to obtain the pointwise-in-time error, the right-hand side can not include $\max$ operator in theorem \ref{thm:fractional_Gronwall_inequality_theorem}. The main reason is that the classical Gr\"onwall inequality requires  $\sum_{j=1}^n \tilde{p}_{n-j}^{(n)}F_{j}$ should be non-decreasing sequence with respect to \( n \). The generic $F_j$ may not satisfy this condition. However, graded mesh or some special mesh establishes its non-decreasing property with the explicit bound of truncated error. 
\end{remark}
\begin{theorem}[pointwise-in-time error]\label{thm:pointwise_error}
On graded or quasi-graded meshes, which denote the discrete-time nodes as $t_{n} \sim T\left( n/N \right) ^{r}$  and its time steps fulfill $t_{n} \sim \tau n^{r-1}$ with the smallest scale $\tau=T/N^{r}$, if
\begin{equation*}
\mathrm{D}_\tau^\alpha \mathcal{E}^n \le \kappa \mathcal{E}^{n}  + \left| \mathtt{T}_{h}^{n} \right|  + \left| \mathtt{T}_{\tau}^{n}  \right|, n=1,2, \cdots, N,
\end{equation*}
where  $\mathcal{E}^{n}$ is the $L^{2}$-norm error on the $n$-th time level and $\mathtt{T}_{h}^{n}, \mathtt{T}_{\tau}^{n}$ are corresponding spacial and temporal truncated errors respectively, then 
\begin{equation}
\mathcal{E}^{n} \lesssim 
\begin{cases}
\tau^{(2-\alpha) / r } t_n^{\beta-(2-\alpha) / r},
& r > \frac{2-\alpha}{1+\beta-\alpha}, \\
\tau^{1+\beta-\alpha} t_n^{\alpha-1} (1+\ln(n)), & r = \frac{2-\alpha}{1+\beta-\alpha}, \\
\tau^{1+\beta-\alpha} t_n^{\alpha-1}, & r <\frac{2-\alpha}{1+\beta-\alpha},
\end{cases}  
\end{equation}
is valid under the hypothesis $\| \partial_t^m u \| \lesssim t^{\beta-m} \text{ for } m =1, 2 \text{ and } \beta \in(0,1) \cup(1,2]$.
\end{theorem}
\begin{proof}
By the result of lemmas \ref{lem:DCS_analysis} and \ref{lem:truncated_error}, we can directly prove that
 \begin{align*}
 & 2\sum_{k=1}^n p_{n-k}^{(n)} a_0^{(k)} \int_{t_{k-1}}^{t_k}\left(t-t_{k-1}\right)\left|u_{t t}\right| \mathrm{d}t \\
 \lesssim & 2 \sum_{k=1}^{n}\int_{t_{k-1}}^{t_{k}}\omega_{\alpha}(t_{n}-s)\mathrm{d}s \cdot \omega_{2-\alpha}(\tau_{k})\cdot \frac{1}{\tau_{k}}\int_{t_{k-1}}^{t_{k}}(t-t_{k-1})t^{\beta-2}\mathrm{d}t \\
 \le & \sum_{k=2}^{n}\omega_{\alpha}(t_{n}-t_{k})\omega_{2-\alpha}(\tau_{k})\tau_{k}^{2}
 t_{k-1}^{\beta-2} + 2\omega_{\alpha}(t_{n}-t_{1})\omega_{2-\alpha}(\tau_{1})\frac{\tau_{1}^{\beta}}{\beta} \\
\lesssim & \sum_{k=1}^{n}(n^{r}-k^{r})^{\alpha-1}k^{(r-1)(3-\alpha)+r(\beta-2)} \\
= & 
\begin{cases}
\mathcal{O}(n^{r\beta-(2-\alpha)}),              & (r-1)(3-\alpha)+r(\beta-2) > -1, \\
\mathcal{O}(n^{r\beta-(2-\alpha)}( 1+ \ln (n))), & (r-1)(3-\alpha)+r(\beta-2) = -1, \\
\mathcal{O}(n^{ r(\alpha-1)}),                    & (r-1)(3-\alpha)+r(\beta-2) < -1,
\end{cases} \\
\sim &
\begin{cases}
\tau^{(2-\alpha) / r } t_n^{\beta-(2-\alpha) / r},
& r > \frac{2-\alpha}{1+\beta-\alpha}, \\
\tau^{1+\beta-\alpha} t_n^{\alpha-1} (1+\ln(n)), & r = \frac{2-\alpha}{1+\beta-\alpha}, \\
\tau^{1+\beta-\alpha} t_n^{\alpha-1}, & r <\frac{2-\alpha}{1+\beta-\alpha}.
\end{cases}
\end{align*}
When $n=1$, it is clear that the second line equals $\mathcal{O}(\tau_{1}^{\beta})$ by direct calculation which consists of the above conclusion. 
\end{proof}
\begin{remark}
The above poinwise-in-time error estimate for pure dissipative system \cite{Kopteva2019} can be obtain more  directly, meanwhile the Gr\"onwal is not required.
\begin{align*}
& \mathrm{D}_\tau^\alpha \mathcal{E}^n \le \left| \mathtt{T}_{h}^{n} \right|  + \left| \mathtt{T}_{\tau}^{n}  \right|, n=1,2, \cdots, N, \\
\Rightarrow ~~~~ & \mathcal{E}^n \le \sum_{j=1}^{n}p_{n-j}^{(j)}\left(\left| \mathtt{T}_{h}^{n} \right|  + \left| \mathtt{T}_{\tau}^{n}  \right|\right)  \le   \sum_{j=1}^{n}\tilde{p}_{n-j}^{(j)}\left(\left| \mathtt{T}_{h}^{n} \right|  + \left| \mathtt{T}_{\tau}^{n}  \right|\right). 
\end{align*}
Consequently, the explicit expression of the right-hand side can be obtained similarly.
\end{remark}
\section{Numerical Verification}
\label{sec:numerical_verification}
In this section, we will provide several typical numerical examples to verify our theoretical analysis results. 
Verify the convergence of the L1 scheme at different times under different graded grid parameters $ r $ with the given decay rate of the solution \( \beta \in (0, 2] \).
\subsection{The growing linear ODE}
To avoid disturbance of spacial discrete error, we consider 
\begin{equation}
\partial_t^{\alpha}u(t)=\kappa u + f(t), \text{ with } u(0)=0, \label{eq:growing linear ODE} 
\end{equation}
with exact solution \( u(t) = \omega_{1+\beta}(t) \) and \( f(t) = \omega_{1+\beta-\alpha}(t) - \kappa \omega_{1+\beta}(t) \) for testing the pointwise-in-time error.
\begin{table}[h]
\centering
  \begin{tabular}{ccccccc}
    \toprule
    \multirow{2}{*}{$N _{\tau}$} &
      \multicolumn{2}{c}{$r=1$} &
      \multicolumn{2}{c}{$r=\frac{2-\alpha}{1+\beta-\alpha}$} &
      \multicolumn{2}{c}{$r=\frac{2-\alpha}{1+\beta-\alpha}+ 1$} \\
      \cline{2-7}
      & {$\|\cdot\|_\infty$} & {c.o.} & {$\|\cdot\|_\infty$} & {c.o.} & {$\|\cdot\|_\infty$} & {c.o.} \\
      \midrule
   64  &  1.335e-01 & - &  3.574e-02 & - &  1.021e-02 & - \\ 
  128  &  1.056e-01 & 0.338 &  2.350e-02 & 0.605 &  5.467e-03 & 0.901 \\ 
  256  &  8.429e-02 & 0.325 &  1.548e-02 & 0.602 &  2.929e-03 & 0.900 \\ 
  512  &  6.771e-02 & 0.316 &  1.021e-02 & 0.601 &  1.570e-03 & 0.900 \\ 
    Theo. c.o. & - &  0.300 & - &  0.600  &  - & 0.900  \\
    \bottomrule
  \end{tabular}
  \caption{The error and convergence order of beginning time level with $\alpha=0.3,\beta=0.6$ for \eqref{eq:growing linear ODE}}.
  \label{tab:beginning_time_ODE}
\end{table}
\begin{table}[h]
\centering
  \begin{tabular}{ccccccc}
    \toprule
    \multirow{2}{*}{$N _{\tau}$} &
      \multicolumn{2}{c}{$r=1$} &
      \multicolumn{2}{c}{$r=\frac{2-\alpha}{1+\beta-\alpha}$} &
      \multicolumn{2}{c}{$r=\frac{2-\alpha}{1+\beta-\alpha}+1$} \\
      \cline{2-7} 
      & {$\|\cdot\|_\infty$} & {c.o.} & {$\|\cdot\|_\infty$} & {c.o.} & {$\|\cdot\|_\infty$} & {c.o.} \\
      \midrule
   64  &  1.653e-01 & - &  2.831e-02 & - &  1.546e-02 & - \\ 
  128  &  1.039e-01 & 0.670 &  1.203e-02 & 1.235 &  5.882e-03 & 1.394 \\ 
  256  &  6.484e-02 & 0.680 &  5.043e-03 & 1.254 &  2.233e-03 & 1.398 \\ 
  512  &  4.027e-02 & 0.687 &  2.092e-03 & 1.269 &  8.466e-04 & 1.399 \\ 
      Theo. c.o. & - &  0.700 & - & 1.400$^-$ &  - & 1.400 \\
    \bottomrule
  \end{tabular}
  \caption{The error and convergence order of ending time level with $\alpha=0.3,\beta=0.6$ for \eqref{eq:growing linear ODE}}.
\label{tab:ending_time_ODE}
\end{table}
From the accurate-order tables \ref{tab:beginning_time_ODE} \ref{tab:ending_time_ODE}, we can know that the factor \( (1+\ln(n)) \) is little effect on the error order at the initial time when \( r = \frac{2-\alpha}{1+\beta-\alpha} \), whereas the large \( n \) will lead to loss of convergence order. 
\subsection{The reaction-diffusion equation}
 The spatial discretization adopts center difference with discrete scale \( h= 2^{-9}\pi \).
Consider $\partial_{t}^{\alpha}u- \Delta u = \kappa u + f(\mathbf{x},t)$ with $\kappa > 0$ and verify the convergence order by setting exact solution $u= \sin(x)\omega_{1+\beta}(t)$ with the corresponding source term $f = \sin(x)(\omega_{1+\beta-\alpha}(t) + (1-\kappa) \omega_{1+\beta}(t))$. For simplicity, let $\kappa=1$, \( \mathbf{x}\in [-\pi, \pi] \), the final time \( T=1 \) and 
\(     \mathcal{E}^{n}(N_\tau):=  \sqrt{h \Sigma_{h}(u_h(t_n)-U_h^n)^2} \).
\begin{table}[h]
\centering
  \begin{tabular}{ccccccc}
    \toprule
    \multirow{2}{*}{$N _{\tau}$} &
      \multicolumn{2}{c}{$\gamma=1$} &
      \multicolumn{2}{c}{$\gamma=\frac{2-\alpha}{1+\beta-\alpha}$} &
      \multicolumn{2}{c}{$\gamma=\frac{2-\alpha}{1+\beta-\alpha}+1$} \\
      \cline{2-7}
      & {$\mathcal{E}^{1}(N _{\tau})$} & {c.o.} & {$\mathcal{E}^{1}(N _{\tau})$} & {c.o.} & {$\mathcal{E}^{1}(N _{\tau})$} & {c.o.} \\
      \midrule
   64  &  2.192e-01 & - &  6.296e-02 & - &  3.374e-02 & - \\ 
  128  &  1.781e-01 & 0.300 &  4.154e-02 & 0.600 &  2.006e-02 & 0.750 \\ 
  256  &  1.446e-01 & 0.300 &  2.740e-02 & 0.600 &  1.193e-02 & 0.750 \\ 
  512  &  1.175e-01 & 0.300 &  1.808e-02 & 0.600 &  7.093e-03 & 0.750 \\ 
    Theo. c.o. & - &  0.300 & - &  0.600 &  - & 0.750  \\
    \bottomrule
  \end{tabular}
  \caption{The error and convergence order of beginning time level with $\alpha=0.3,\beta=0.6$.}
  \label{tab:beginning_time}
\end{table}
\begin{table}[h]
\centering
  \begin{tabular}{ccccccc}
    \toprule
    \multirow{2}{*}{$N _{\tau}$} &
      \multicolumn{2}{c}{$r=1$} &
      \multicolumn{2}{c}{$r=\frac{2-\alpha}{1+\beta-\alpha}$} &
      \multicolumn{2}{c}{$r=\frac{2-\alpha}{1+\beta-\alpha}+1$} \\
      \cline{2-7} 
      & {$\mathcal{E}^{N_\tau}(N _{\tau})$} & {c.o.} & {$\mathcal{E}^{N_\tau}(N _{\tau})$} & {c.o.} & {$\mathcal{E}^{N_\tau}(N _{\tau})$} & {c.o.} \\
      \midrule
   64  &  4.528e-02 & - &  8.350e-03 & - &  6.046e-03 & - \\ 
  128  &  2.783e-02 & 0.702 &  3.480e-03 & 1.263 &  2.337e-03 & 1.371 \\ 
  256  &  1.711e-02 & 0.702 &  1.433e-03 & 1.280 &  8.934e-04 & 1.387 \\ 
  512  &  1.052e-02 & 0.701 &  5.816e-04 & 1.301 &  3.354e-04 & 1.413 \\
      Theo. c.o. & - &  0.700 & - & 1.400$^{-}$ &  - & 1.400 \\
    \bottomrule
  \end{tabular}
  \caption{The error and convergence order of ending time level with $\alpha=0.3,\beta=0.6$.}
\label{tab:ending_time}
\end{table}

\section{Conclusions}
In this article, we answered an open question about discrete complementary convolution kernels in \cite{Yang2022} and demonstrated the asymptotic estimation of DCC under general non-uniform grids using a relatively concise proof. By utilizing the continuous counterpart of DCC, we proved the asymptotically compatible Gr\"onwall inequality and applied it to the numerical analysis of time fractional parabolic equations. The required estimation formula for DCS in a graded mesh provides rigorous theoretical proof and illustrates the pointwise error estimation in graded mesh.The strict proof of the necessity for the Caputo discrete coefficient to hold true for $\mathrm{D}_{\tau} ^ {\alpha}  | u ^ {n}  | ^ {2}  \le (\mathrm{D}_{\tau} ^ {\alpha} u ^ {n}, u ^ {n}) $ is also given in this paper incidentally. For the stability analysis of higher-order Caputo BDF\( _k \), a unified analysis framework will be provided in the subsequent work.
\appendix
\section{The estimate of DCS}
\label{sec:appendix_dcs}
\par In this part, we need to propose a beneficial estimate for DCS. Since the definition of Caputo fractional derivatives is based on the convolution formula, the estimation of discrete convolution integrals has been mentioned in many kinds of literature \cite{Lubich1986,Stynes2017,Gracia2017}. However, the estimation methods they use are relatively complicated, and the results are not very concise \cite{Stynes2017,Stynes2023}. Here, we want to obtain fairly precise and easily analyzed results. Compared with Jin et al \cite[Lemma 3.11]{Jin2023}, our result can be regarded as its graded mesh version.
\begin{lemma}[the  estimate of DCS]\label{lem:DCS_analysis}
For arbitrary positive real number $r$, the discrete convolution summation $S_{r,p,q}^{(n)} :=\sum \limits_{k=1}^{n-1}(n^{r}-k^{r})^p k^q $ have an estimate 
\begin{equation*}
S_{r,p,q}^{(n)} \sim 
\begin{cases}
\mathcal{O}\left(n^{rp+q+1}\right),                  & \min \{p,q\} > -1, \\
\mathcal{O}\left(n^{rp+q+1}(1+\ln(n))\right),        & \min \{p,q\} = -1,  \\
\mathcal{O}\left(n^{\max \{rp, (r-1)p+q\} } \right), & \min \{p,q\} < -1.
\end{cases}
\end{equation*}
\end{lemma}
\begin{proof}\label{prf:estimate of DCS}
To begin with the proof, we want to explain the main idea of analysis. If $S_{1,p,q}^{(n)}$ have equal and invariant convergence rate with $n^{C_{1,p,q} }$ when $n\to + \infty $, where the $C_{1,p,q} $ is a fixed constant, then
\begin{equation*}
\lim _{n \rightarrow \infty} \frac{S_{1, p , q}^{(n)}}{n^{C_{1,p,q} }} \sim 1.
\end{equation*}
Thus, we should find the highest-order terms of $ S_{1,p,q}^{(n)} $. Even though each term of $S_{1,p,q}^{(n)}$ has a different order, we can ensure the range of its. Moreover, the highest order and
the lowest order should come from the first term and the last term, or inverse ones which is the feature of finite terms polynomials. Without loss of generality, we just need to consider the following four cases.
\begin{enumerate}
\item When $p > -1$ and $q>-1$
\begin{align*}
\lim _{n \rightarrow \infty} \frac{S_{1, p , q}^{(n)}}{n^{p+q+1}} = \lim _{n \rightarrow \infty} \sum_{k=1}^{n-1}\left(1-\frac{k}{n}\right)^p \cdot\left(\frac{k}{n}\right)^q \cdot \frac{1}{n} = B(p+1, q+1) \sim 1.
\end{align*}
Hence, $ S_n^{(1, p , q)} \sim n^{p+q+1} $ while $p > -1$ and $q>-1$.
\item When $p > -1 = q$
\begin{align*}
\frac{S_{1,p,-1}^{(n)}}{n^{p}(1+ \ln n)} = & \frac{1}{(1+ \ln n) } \sum_{k=1}^{n-1} \left( 1- \frac{k}{n} \right)^{p} \frac{1}{k} \\
= & \frac{1}{(1+ \ln n) } \sum_{k=1}^{n-1} \left( 1- \frac{k}{n} \right)^{p+1} \left( \frac{1}{n-k} + \frac{1}{k} \right).
\end{align*}
The $\dfrac{1}{n^{p+1}} <\left( 1- \dfrac{k}{n} \right)^{p+1} < 1$ with $p+1>0$,	and the equality$\sum\limits_{k=1}^{n-1} \frac{n}{(n-k)k}   = 2 \sum\limits_{k=1}^{n-1} \frac{1}{k}$,	we know that
\begin{align*}
& \frac{1}{(1+ \ln n)} \frac{2\gamma}{n^{p+1} } \left( \frac{1}{\gamma} ( \sum_{k=1}^{n} \frac{1}{k} - \ln n ) - \frac{1}{\gamma n} + \ln n
\right) \\
< & \frac{1}{(1+ \ln n)} \frac{2}{n^{p+1} } \left( ( \sum_{k=1}^{n} \frac{1}{k} - \ln n ) - \frac{1}{n} + \ln n
\right) \\
= & \frac{1}{(1+ \ln n)} \frac{2}{n^{p+1} } \sum_{k=1}^{n-1} \frac{1}{k} < \frac{S_n^{(p,-1)}}{n^{p}(1+ \ln n) } < \frac{2}{(1+ \ln n)} \sum_{k=1}^{n-1} \frac{1}{k} \\
= & \frac{2}{(1+ \ln n) } \left(( \sum_{k=1}^{n-1} \frac{1}{k} - \ln n ) + \ln n
\right)\\
< & \frac{2}{(1+ \ln n) } \left( \frac{1}{\gamma} ( \sum_{k=1}^{n} \frac{1}{k} - \ln n ) + \ln n
\right),
\end{align*}
where $\sum_{k=1}^{n} \frac{1}{k}-\ln n > 0 $ and
$\lim\limits _{n \rightarrow \infty}(\sum_{k=1}^{n} \frac{1}{k}-\ln n )$
converge to the Euler constant $\gamma\approx 0.57721<1$.
Take the limit $ n \to \infty  $, then
\begin{equation*}
0 < \lim _{n \rightarrow \infty} \frac{S_{1, p,-1}^{(n)}}{n^{p}(1+ \ln n) } < 2.
\end{equation*}
If $ S_{1, p,-1}^{(n)} \nsim  n^{p}(1+ \ln n) $, then the limitation converges to either $ 0 $   or $ \infty $. For this reason, $ S_{1, p , q}^{(n)} \sim n^{p}(1+ \ln n)  $ while $p = -1$ or $ q=-1$.
\item When $ -1 > p \geq q $
\begin{align*}
\frac{S_{1,p,q}^{(n)}}{n ^p} = & \sum_{k=1}^{n-1} \frac{1}{n^{p}} k^{p}(n-k)^{q}  = \sum_{k=1}^{n-1} \frac{1}{n^{p}} k^{p}(n-k)^{p}(n-k)^{q-p} \\
< & \sum_{k=1}^{n-1} \left( \frac{n}{k(n-k)} \right) ^{-p} = \sum_{k=1}^{n-1} \left( \frac{1}{k} + \frac{1}{n-k} \right) ^{-p} \\
= & \sum_{k=1}^{\left[ \frac{n}{2} \right] -1} \left( \frac{1}{k} + \frac{1}{n-k} \right) ^{-p} +\sum_{k=\left[ \frac{n}{2} \right] }^{n-1} \left( \frac{1}{k} + \frac{1}{n-k} \right) ^{-p} \\
< &  2^{-p} \sum_{k=1}^{\left[ \frac{n}{2}\right] -1} \left( \frac{1}{k} \right) ^{-p} + 2^{-p} \sum_{k=\left[ \frac{n}{2}\right]}^{n -1} \left( \frac{1}{n-k} \right) ^{-p} \le 2^{1-p}  \sum_{k=1}^{\left[ \frac{n}{2}\right]} \left( \frac{1}{k} \right) ^{-p},
\end{align*}
then
\begin{equation*}
0 < \lim _{n \rightarrow \infty} \frac{S_{1,p,q}^{(n)}}{n^{p}} < 2^{1-p} \sum_{k=1}^{\infty} \left( \frac{1}{k} \right) ^{-p}.
\end{equation*}
\item When $ p > -1 >q$
\begin{align*}
\frac{S_{1,p,q}^{(n)}}{n ^p} \le & \frac{\sum\limits_{k=1}^{n-1}k^{p}(n-k)^{q}}{n^{p}} \le  \frac{\left(\sum\limits_{k=1}^{n-1}k^{\mu p}\right)^{\frac{1}{\mu}}\cdot \left(\sum\limits_{k=1}^{n-1}(n-k)^{\nu q}\right)^{\frac{1}{\nu}}}{n^{p}} \\ 
\lesssim & \left(\sum_{k=1}^{n-1}\left(\frac{k}{n}\right)^{\mu p}\right)^{\frac{1}{\mu}} \cdot \left(\sum\limits_{k=1}^{n-1}(n-k)^{\nu q}\right)^{\frac{1}{\nu}},
\end{align*}
where $\frac{1}{\mu} + \frac{1}{\nu}=1$ and the H\"older inequality have been used in second line.
Let $\mu \to \infty$, namely corresponding $\nu \to 1$, since the classical limit formula and the rest part is convergent series with \( q < -1 \), thence 
\begin{align*}
\frac{S_{1,p,q}^{(n)}}{n ^p} \lesssim 1.
\end{align*}
\end{enumerate}
\par The first two cases can imitate the above proof. Now, we consider the last one.
$(n^{r} - 1)^{p} \sim n^{rp}$ and $(n^{r}-(n-1)^{r})^{p}(n-1)^{q} \sim n^{(r-1)p+q}$ are terms while $k=1$ and $k=n-1$ respectively. To achieve the whole convergence rate, we just need to find the highest order of this finite sum about $n$. The range of this order about $n$ is from $n^{rp}$ to $n^{(r-1)p+q}$, so the convolution summation is the same order with $n^{\max \{rp, (r-1)p+q\}}$. 	So far, the whole proof is accomplished.
\end{proof}
Here, the Matlab code can be tested that our estimate is optimal,
\begin{lstlisting}[style=Matlab-editor]
N = 16;
q = zeros(N,1);
for j = 1:N
    q(j) = computing(2^j, -1, 2,.3); 
    % close to a constant
end
function q = computing(n,p,q,r)
a = sum((n^r - (1:n-1).^r).^p.*(1:n-1).^q);
if min(p,q) > -1
    b = n^(r*p+q+1);
elseif min(p,q) == -1
    b = n^(r*p+q+1)*(1+log(n));
elseif min(p,q) < -1
    b = n^(max(r*p,(r-1)*p+q));
end
q = a/b;
end
\end{lstlisting}
\section{A proof of necessity}\label{sec:appendix_necessary_condition}
In classical parabolic problems, $\tfrac{1}{2} \partial_t \| u \|^2 = ( \partial_t u, u )$ play a significant role for $L ^{2}$-stability analysis. In time-fractional Caputo cases, A. A. Alikhanov \cite[Lemma 1]{Alikhanov2010} has proved an analog $\frac{1}{2}\partial_{t}^{\alpha} \left\|u^{n}\right\|^{2} \leq \left( \partial_{t}^{\alpha}u^{n}, u^{n} \right)$.
Considering discrete version  with proper discretization,  as we well know, the DC kernels $a _{n-k}^{(n)}$ monotonically increasing with respect to $k$ can lead to the following inequality in $L ^{2}$-stability analysis 
\begin{equation}
\frac{1}{2}\mathrm{D} _{\tau}^{\alpha} \|u ^{n}\| ^{2} \le (\mathrm{D} _{\tau}^{\alpha}u ^{n}, u ^{n}) \label{eq:discrete_l2_stability}
\end{equation}
which includes L1-scheme,  L21$_{\sigma}$-scheme with additional condition \cite{Liao2021} and so on. 
\par In this paper, we will propose rigorous proof the monotonicity of DC kernels is not only a sufficient condition of relation \eqref{eq:discrete_l2_stability} but also a necessary condition. However, some high-order schemes, e.g. L1$^{+}$ or Caputo BDF-k do not have a relation \eqref{eq:discrete_l2_stability}, which is the main difficulty of stability analysis.
The discrete $L ^{2}$-norm estimation is essentially a series of inequalities established by quadratic form. The proof in this section also starts from this perspective, establishing the relationship between the positive inertia index of the representation matrix of the quadratic form and the uniform monotonicity.
\begin{lemma}
\label{lem:eigenvalues_analysis}
For arbitrary positive real sequence $\{d_{i}\}_{i=1}^{n}$, define the real symmetrical matrix value function $\mathbf{M}(\mathbf{d}_{n}):=\mathbf{L}_{n}\mathbf{d}_{n} +  \mathbf{d}_{n} \mathbf{L}_{n}^{\mathrm{T}} - \mathbf{d}_{n}$, whose $\mathbf{d}_{n}=\operatorname{diag}(d_{1}, d_{2}, \ldots, d_{n})$  and $\mathbf{L}_{n}$ is the all-one unite lower triangular matrix,
then the matrix representation of quadratic form  $\mathbf{M}(\mathbf{d}_{n})$ is positive definite if only if  $\{d_{i}\}_{i=1}^{n}$ must be uniformly strictly monotonically increasing.
\end{lemma}
\begin{proof}
Denote characteristic polynomial $p(\lambda; \mathbf{d}_{n}) := \left|\mathbf{M}(\mathbf{d}_{n}) - \lambda \mathbf{I}_{n} \right| $ and check its expression
\begin{align*}
p(\lambda; \mathbf{d}_{n})  = 
\begin{vmatrix}
d_{1}-\lambda & d_{1} & \ldots & d_{1} \\
d_{1} & d_{2}-\lambda & \ldots & d_{2} \\
\vdots & \vdots & \ddots & \vdots \\
d_{1} & d_{2} & \ldots & d_{n}-\lambda \\
\end{vmatrix}
\xlongequal{\textbf{step 1} } 
\begin{vmatrix}
1 & 0 & \ldots & 0  \\ 
-d_{1} &  d_{1}-\lambda & \ldots & d_{1} \\
\vdots &  \vdots & \ddots & \vdots \\
-d_{n} & d_{1} & \ldots & d_{n}-\lambda \\
\end{vmatrix} \\
\xlongequal{\textbf{step 2} }
\begin{vmatrix}
1 & 1 & \ldots & 1 \\ 
-\delta_{0}^{1} & -\lambda & \ldots & 0 \\
\vdots &  \vdots & \ddots & \vdots \\
-\delta_{0}^{n} & -\delta ^{n}_{1} & \ldots & -\lambda \\
\end{vmatrix}
\xlongequal{\textbf{step 3}}
\begin{vmatrix}
x _{0}^{(n)} & 0 & \ldots & 0 \\ 
-\delta_{0}^{1} & -\lambda & \ldots & 0 \\
\vdots &  \vdots & \ddots & \vdots \\
-\delta_{0}^{n} & -\delta ^{n}_{1} & \ldots & -\lambda \\
\end{vmatrix} = (-\lambda)^{n}x_{0}^{(n)},
\end{align*}
where $\delta ^{k}_{j}:=d_{k}-d_{j}$ for all $ 0 \le j < k \le n $ with letting $d _0=0$.
\begin{itemize}
\item[\textbf{step 1}]
Add one row and one column to the determinant and ensure that the values of the determinant do not change, and mark the added rows and columns as the zeroth row and zeroth column for easy description in the next;
\item[\textbf{step 2}]
Add the zeroth column of the determinant to each subsequent column and denote $d_k=\delta_{0}^{k}$ for all $1\le k \le n$ as a unified form;
\item[\textbf{step 3}]
Retain the first element of the zeroth row in the determinant and reduce the remaining $1$ to $0$ through the following process. Firstly, by multiplying the elements in the $n$th row by $1/ \lambda $ and adding them to the zeroth row, the elements in the $n$th column of the zeroth row are reduced to zero. At the same time, the $k$th element of the zeroth row is reduced to $1-  \tfrac{1}{\lambda}\delta_ {k} ^ {n}, \forall~0  \le k  \le n $. This step is completed through a similar process for subsequent rows.
\end{itemize}
In order to provide a more detailed description of  \textbf {step 3}, which is the process of changing the elements in the zeroth row of the determinant, we record $x_ {k} ^ {(n)} $ as the $k$th element in the zeroth row, and the superscript represents the $n$th change, such as $x_ {k} ^ {(0)}  \equiv 1,  \forall~0  \le k \le n $ indicates that each element in the zeroth row in the initial state is one. So, the entire process is described as follows,
\begin{align*}
\begin{array}{ll}
\text{For all } 0\le k \le n-0, & x_{k}^{(0)} = 1; \\  
\text{For all } 0\le k \le n-1, & x_{k}^{(1)} = x_{k}^{(0)} - \tfrac{1}{\lambda}\delta ^{n}_{k}x_{n-0}^{(0)};\\
~~~~~~~~~~~~~~~~ \vdots & ~~~~~~~ \vdots \\
\text{For all } 0\le k \le 0, & x_{k}^{(n)}   = x_{k}^{(n-1)} - \tfrac{1}{\lambda}\delta ^{1}_{k}x_{1}^{(n-1)};\\
\end{array}
\end{align*}
In fact, each step in the above process contains this linear transformation, and the coincidence of multiple linear transformations can be represented by corresponding matrix multiplication. So this process is represented as a matrix form,
\begin{equation}
x_{0}^{(n)} = \left(\prod _{k=1}^{n}[\mathbf{I}_{k}, \mathbf{v}_{k}]\right)\mathbf{e}_{n+1}, \label{eq:x0_compact_form}
\end{equation}
where $\left[\mathbf{I}_{k}, \mathbf{v}_{k}\right]$ represent  $k\times k$ identity matrix stitching a  $k$-dimensional column vector $\mathbf{v}_{k}:=-\frac{1}{\lambda}\left[\delta ^{k}_{0}, \delta_{1}^{k} \ldots, \delta ^{k}_{k-1}\right]^{\mathrm{T}}$, e.g.
\[
\left[\mathbf{I}_{k}, \mathbf{v}_{k}\right] = 
\begin{bmatrix}
    1 & 0 & \ldots & 0 & -\frac{1}{\lambda}\delta ^{k}_{0} \\
    0 & 1 & \ldots & 0 & -\frac{1}{\lambda}\delta ^{k}_{1} \\
    \vdots & \vdots & \ddots & \vdots & \vdots  \\
    0 & 0 & \ldots & 1 & -\frac{1}{\lambda}\delta ^{k}_{k-1} \\
\end{bmatrix}_{k\times(k+1)}.
\]
and $\mathbf{e}_{n+1}$ is $(n+1)$-dimensional all-one column vector.
\par Next, we will calculate the specific expression of $(-\lambda)^{n}\left(\prod _{k=1}^{n}[\mathbf{I}_{k}, \mathbf{v}_{k}]\right)\mathbf{e}_{n+1}$. Denote $\gamma _{k-i}^{k} := - \tfrac{1}{\lambda}\delta _{k-i}^{k}$ for all $1\le i \le k$ and observe its non-compact form 
\begin{equation*}
[\mathbf{I}_{1}, \mathbf{v}_{1}] [\mathbf{I}_{2}, \mathbf{v}_{2}] \ldots [\mathbf{I}_{n}, \mathbf{v}_{n}] \mathbf{e}_{n+1}.
\end{equation*}
From right to left is a matrix representation of the change process of the elements in the zeroth row of the determinant in  \textbf {step 3}, whose dimensions gradually decrease and eventually change to a one-dimensional number. From left to right will be our next calculation process, where the dimension of the row vector is gradually increased to $n+1$-dimension, but ultimately multiplying it with the column vector is equivalent to summing the $n+1$ components of the row vector, which also yields a one-dimensional number.
\par At this point, our problem transforms into determining each element in the row vector $ \prod_ {k=1} ^ {n} [\mathbf{I}_{k}, \mathbf {v}_ {k}]$. Here, we will approach this computational problem from the perspective of combinatorial mathematics. To this end, we first explain the meaning of the following symbols in combinatorial mathematics
\begin{align*}
\gamma _{i _{j-1}}^{i _{j}}: & \text{ a single step  from } i _{j-1} \text{ to } i _{j}, \\
\prod _{j=0}^{k-1} \gamma _{i _{j}} ^{i _{j+1}}: & \text{ a multi-step from } i _{0} \text{ to } i _{k} \text{ via } k\text{ steps}.
\end{align*}
Particularly, $ \gamma _{0}^{0} = 1$ means the position has no movement.
\par Next, we use mathematical induction to prove that $\forall~1  \le i  \le n+1 $ has
\begin{equation}
\left\langle \prod _{k=1}^{n} [\mathbf{I}_{k}, \mathbf{v}_{k}] \right\rangle _{i} = \sum _{k=1}^{n+1-i}\prod _{j=0}^{k-1}\gamma _{i _{j}}^{i _{j+1}}, \label{eq:vector_element}
\end{equation}
where $ 0=i_ {0}<i_ {1} <i_ {2} < \ldots <i_ {k} =n+1-i $ and $ \langle  \cdot  \rangle_ {i} $ represents the $i$-th component in the vector. In the above equation, the specific meaning of its abstract expression is the total sum of all possibilities which from $i_ {0}=0 $ departs, can only move in the direction of point $n+1-i $ each time, but does not limit the forward step size, and ultimately reaches point $n+1-i$.
\par When $n=1$, $[1, \gamma_{0}^{1}]=[\mathbf{I}_{1}, \mathbf{v}_{1}]$, obviously holds. Assume \eqref{eq:vector_element} holds for $n$, and now, we will prove $n+1$ also holds. 
\begin{align*}
& \left(\prod _{k=1}^{n} [\mathbf{I}_{k}, \mathbf{v}_{k}] \right)\cdot [\mathbf{I}_{n+1}, \mathbf{v}_{n+1}] = \left[\prod _{k=1}^{n} [\mathbf{I}_{k}, \mathbf{v}_{k}], \left(\prod _{k=1}^{n} [\mathbf{I}_{k}, \mathbf{v}_{k}] \right)\mathbf{v}_{n+1} \right] \\
= & \left[\prod _{k=1}^{n} [\mathbf{v}_{k}, \mathbf{I}_{k} ], \sum_{i=1}^{n+1}\left(\sum _{k=1}^{n+1-i}\prod _{j=0}^{k-1}\gamma _{i _{j}}^{i _{j+1}}\right)\gamma_{n+1-i}^{n+1}\right]  = \left[\prod _{k=1}^{n} [\mathbf{v}_{k}, \mathbf{I}_{k} ], \sum _{k=1}^{n+1}\prod _{j=0}^{k-1}\gamma _{i _{j}}^{i _{j+1}}\right],
\end{align*}
where the calculation of the last step can also start from the meaning of combination, for all $1  \le i  \le n+1 $,
\begin{align*}
\left(\sum _{k=1}^{n+1-i}\prod _{j=0}^{k-1}\gamma _{i _{j}}^{i _{j+1}}\right)\gamma_{n+1-i}^{n+1},
\end{align*}
represent all possibilities from $i _{0}=0$ depart, reach point $n+1-i $, and continue moving forward to reach point $n+1 $, thus consequently $ \sum_ {i=1} ^ {n+1} $ operation means that taking the sum of all the possibilities of starting from $ i_0=0 $ and reaching point $n+1 $, which can be used to directly write the final result based on this meaning.
\par By the above analysis, we know that 
\begin{equation}
x_0^{(n)}=\sum_{i=1}^{n+1}\sum _{k=1}^{n+1-i}\prod _{j=0}^{k-1}\gamma _{i _{j}}^{i _{j+1}}, \label{eq:x0_final}
\end{equation}
those value is represented as from $i_ 0=0 $ depart and take $0  \le k  \le n $ steps to reach all possible sums up to point $n$. Based on $ \gamma_ {k-i} ^ {k}:=-  \frac{1}{\lambda}\delta_ {k-i} ^ {k} $ is the $-1$ degree of $-  \lambda $, then $ \prod_ {j=0} ^ {k-1}  \gamma_ {i_ {j}} ^ {i_ {j+1}} $ is represented as the $- k $ degree of $-  \lambda $. All $- k $ degrees in the expression \eqref{eq:x0_final} are represented as from $i_ 0=0 $ depart, after $k $ steps, and reach all possibilities up to point $n $, i.e
\begin{equation*}
    \sum\limits_{1\le i_{1} < \ldots < i_{k} \le n}\prod\limits_{j=0}^{k-1}\gamma_{i_{j}}^{i_{j+1}}.
\end{equation*}
Then 
\begin{align*}
p(\lambda; \mathbf{d}_{n}) = (-\lambda)^{n}\sum _{k=1}^{n}\left(\sum\limits_{1\le i_{1} < \ldots < i_{k} \le n}\prod\limits_{j=1}^{k}\gamma_{i_{j-1}}^{i_{j}}\right) = \sum _{i=0}^{n} C _{i}(- \lambda)^{i},
\end{align*}
where $ C_ {i} $ represents both the coefficient of the degree i of $-  \lambda $ and from $i_0=0 $ depart at, take $n-i $ steps, and reach all possibilities up to point $n$, namely $ C_ {i} $ have specific expression
\begin{align*}
\begin{bmatrix}
C _{0}\\
\vdots \\
C _{i}\\
\vdots \\
C _{n-2}\\
C _{n-1}\\
C _{n}
\end{bmatrix} 
=
\begin{bmatrix}
\prod\limits_{k=1}^{n} \delta ^{k}_{k-1} \\
\vdots \\
\sum\limits_{1\le k_{1} < \ldots < k_{i} \le n}\prod\limits_{j=1}^{i}\delta_{k_{j-1}}^{k_{j}} \\
\vdots \\
\sum\limits_{1 \le k_{1} < k_{2} \le n} \delta_{k_{1}}^{k_{2}} \delta_{0}^{k_{1}}   \\
\sum\limits_{1\le k\le n}\delta_{0}^{k} \\
1
\end{bmatrix} 
\xlongequal{\star}
\begin{bmatrix}
\prod\limits_{k=1}^{n} \lambda_{k} \\
\vdots \\
\sum\limits_{1\le k_{1} < \ldots < k_{i} \le n}\prod\limits_{j=1}^{i} \lambda_{k_{j}}\\
\vdots \\
\sum\limits_{1\le k_{1} < k_{2} \le n}\lambda_{k_{1}}\lambda_{k_{2}} \\
\sum\limits_{1\le k \le n} \lambda_{k} \\
1
\end{bmatrix}
\end{align*}
According to all eigenvalues of a real symmetric matrix are real numbers and assume the $k$th real eigenvalue is $ \lambda_ {k} $, then its characteristic polynomial can be also represented as $C \prod_ {k=1} ^ {n} ( \lambda_ {k}- \lambda)=p ( \lambda;  \mathbf {d}_ {n} ) $ where $C\neq 0$. In this problem, by comparing the coefficient of highest-order term $(-\lambda) ^{n}$, it can be determined that $C=1$. For more coefficient comparisons, we can refer to the identity ($ \star $).
\par Therefore, by comparing the coefficient of zeroth polynomial term, we can conclude that
\begin{equation*}
\prod\limits_{k=1}^{n} \lambda_{k} \equiv \prod\limits_{k=1}^{n}(d _{k} - d _{k-1}), \text{ for all } 1 \le k \le n,
\end{equation*}
namely, for all $1\le k \le n$ have $\lambda  _{k}>0 \Leftrightarrow d _{k} > d _{k-1}$. This completes the proof.
\end{proof}
\begin{theorem}
For all $1 \le n \le N$,
\begin{equation*}
\frac{1}{2}\sum_{k=1}^{n}a_{n-k}^{(n)}\nabla_{\tau} \left\|u^{k}\right\|^{2} \le \left( \sum_{k=1}^{n}a_{n-k}^{(n)}\nabla_{\tau} u^{k}, u^{n} \right),
\end{equation*}
if and only if the DC kernels $a _{n-k}^{(n)}$ is monotonically increasing  with respect to index $k$.
\end{theorem}
\begin{proof}
\begin{align*}
& \left( \sum_{i=1}^{n} a_{n-i}^{(n)}\nabla_{\tau} u^{i}, u^{n} \right) - \frac{1}{2}\sum_{i=1}^{n} a_{n-i}^{(n)}\nabla_{\tau} \left\|u^{i}\right\|^{2} 
= \sum_{i=1}^{n} a_{n-i}^{(n)}\left( \nabla_{\tau} u^{i}, u^{n} - u^{i-\frac{1}{2}}\right) \\
= & \sum_{i=1}^{n} a_{n-i}^{(n)}\left( \nabla_{\tau} u^{i}, u^{n} - u^{i-1}\right)  + \sum_{i=1}^{n} a_{n-i}^{(n)}\left( \nabla_{\tau} u^{i}, u^{i-1} -  u^{i-\frac{1}{2}}\right) \\
= & \sum_{i=1}^{n} a_{n-i}^{(n)}\left( \nabla_{\tau} u^{i}, \sum_{j=i}^{n} \nabla_{\tau} u^{j} \right) -  \frac{1}{2}\sum_{i=1}^{n} a_{n-i}^{(n)}\left\| \nabla_{\tau} u^{i}\right\|^{2} \\
= & \sum_{j=1}^{n} \left(\sum_{i=1}^{j}  a_{n-i}^{(n)} \nabla_{\tau} u^{i}, \nabla_{\tau} u^{j}\right) -  \frac{1}{2}\sum_{i=1}^{n} a_{n-i}^{(n)}\left\| \nabla_{\tau} u^{i}\right\|^{2}\\
= & \mathbf{v}_{n}^{\mathrm{T}}\left(\mathbf{L}_{n}\mathbf{A}_{\mathrm{diag}}\right)\mathbf{v}_{n} - \frac{1}{2}\mathbf{v}_{n}^{\mathrm{T}}\mathbf{A}_{\mathrm{diag}}\mathbf{v}_{n} 
 = \frac{1}{2} \mathbf{v}_{n}^{\mathrm{T}}\left(\mathbf{L}_{n}\mathbf{A}_{\mathrm{diag}} + \mathbf{A}_{\mathrm{diag}}\mathbf{L}_{n}^{\mathrm{T}} - \mathbf{A}_{\mathrm{diag}}\right)\mathbf{v}_{n}, 
\end{align*}
where diagonal matrix $\mathbf{A}_{\mathrm{diag}} := \operatorname{diag}([a_{n-j}^{(n)}]_{j=1}^{n},)$ and column vector $\mathbf{v}_{n} := [\nabla_{\tau}u^{j}]_{j=1}^{n}$.
According to the lemma \ref{lem:eigenvalues_analysis}, $\mathbf{M}(\mathbf{A}_{\mathrm{diag}})$ is a matrix representation of the positive definite quadratic form based on the monotonicity of $a _{n-k}^{(n)}$. Therefore, for arbitrary real vector $\mathbf{v}_{n}$, the sufficient and essential condition to ensure  $\mathbf{v}_{n}^{\mathrm{T}}\mathbf{M}(\mathbf{A}_{\mathrm{diag}})\mathbf{v}_{n}\ge 0$ is that $a _{n-k}^{(n)}$ should be monotonically increasing for all $1\le k\le n$. The assertion have been finished.
\end{proof}

\newpage

%
%
\begin{acknowledgements}
The work is supported by the National Natural Science Foundation of China (12171385) and the Fundamental Research Funds for the Central Universities, PR China (No. xzy022023024).
\end{acknowledgements}
\section*{Conflict of interest}
The authors declared that they have no conflicts of interest to this work.

\bibliographystyle{spmpsci}      
\bibliography{FC.bib}   


\end{document}